\documentclass{amsart}
\usepackage{amsmath,amssymb,amsthm,amsfonts,verbatim}

\DeclareMathSymbol{\twoheadrightarrow} {\mathrel}{AMSa}{"10}

\def\Q{{\mathbb Q}}

                                              \def\mm{{\mathfrak m}}

\def\Z{{\mathbb Z}}

\def\Gal{\mathrm{Gal}}

                          \def\tors{\mathrm{tors}}

\def\Aut{\mathrm{Aut}}

\def\fchar{\mathrm{char}}

        \def\K_a{\bar{K}}

\def\Gu{{\rm G}_}

\def\dim{\mathrm{dim}}
                                              \def\Pic{\mathrm{Pic}}

\def\K{{\mathcal{K}}}

\def\bet{{\mathbf b}}

\newtheorem{thm}{Theorem}[section]

\newtheorem{lem}[thm]{Lemma}

\newtheorem{cor}[thm]{Corollary}

\newtheorem{prop}[thm]{Proposition}

\theoremstyle{definition}

\newtheorem{ex}[thm]{Example}

\newtheorem{rem}[thm]{Remark}

\newtheorem{sect}[thm]{}

\title[Odd-dimensional cohomology]{Odd-dimensional cohomology with finite coefficients and roots of unity}

\author[Yuri G.\ Zarhin]{Yuri G.\ Zarhin}
\thanks{This work was partially supported by a grant from the Simons Foundation (\#246625 to Yuri Zarkhin).}

\address{Department of Mathematics, Pennsylvania State University,
University Park, PA 16802, USA}

\email{zarhin\char`\@math.psu.edu}

\begin{document}

\begin{abstract}
We prove that the triviality of the Galois action on the suitably twisted {\sl odd-dimensional} \'etale cohomology group with finite coefficients  of an absolutely irreducible smooth projective variety implies the existence of certain primitive roots of unity in the field of definition of the variety. This text was inspired  by an exercise in Serre's Lectures on the Mordell--Weil  theorem.
\end{abstract}

\maketitle

\section{Introduction}
\label{intro}

We  recall some basic facts about {\sl cyclotomic characters}.
Let $K$ be a field, $\bar{K}$ its algebraic closure,
$\Gu K=\Aut(\bar{K}/K)$ the absolute Galois group of $K$. Let $n$
be a positive integer  that is {\sl not} divisible by $\fchar(K)$.
We write $\mu_n \subset \bar{K}$ for the
cyclic multiplicative group of $n$th roots of unity in $\bar{K}$.  We write
$$\bar{\chi}_n: \Gu K \to \Aut(\mu_n)=(\Z/n\Z)^{*}$$
for the cyclotomic character that defines the Galois action on $n$th roots of unity. Clearly,
$\mu_n \subset K$ if and only if 
$$\bar{\chi}_n(g)=1 \ \forall g \in \Gu K.$$
Recall that the order of $(\Z/n\Z)^{*}$ is $\phi(n)$ where $\phi$ is the {\sl Euler function}. This implies that
$${\bar{\chi}_n}^{\phi(n)}(g)=1 \ \forall g \in \Gu K.$$

Let $K(\mu_n) \subset \bar{K}$ be the $n$th cyclotomic extension of $K$. Then the 
degree $[K(\mu_n):K]$ of the abelian extension $K(\mu_n)/K$ coincides with the order of the finite commutative Galois  group $\Gal(K(\mu_n)/K).$  By definition of $\bar{\chi}_n$, its
 kernel  coincides with $\Gu {K(\mu_n)}$ and $\bar{\chi}_n$ is the composition of the surjection
$$\Gu K\mapsto \Gu K/\Gu {K(\mu_n)}= \Gal(K(\mu_n)/K)$$
and the   embedding
$$\Gal(K(\mu_n)/K)=\hookrightarrow (\Z/n\Z)^{*},$$
which we continue to denote by  $\bar{\chi}_n$, slightly abusing notation.

\begin{rem}
\label{even}
Clearly, the {\sl exponent} $\exp(n,K)$   of  $\Gal(K(\mu_n)/K)$  divides the order of $\Gal(K(\mu_n)/K)$, which, in turn, divides $\phi(n)$. In addition, if $f$ is an integer then the character $\bar{\chi}_n^f$ is {\sl trivial} if and only if $f$ is divisible by $\exp(n,K)$.  In particular, the characters $\bar{\chi}_n^{\phi(n)}$ and $\bar{\chi}_n^{\exp(n,K)}$ are {\sl trivial}. On the other hand,
if the degree of the  extension $K(\mu_n)/K$ is {\sl even} then  so is $\exp(n,K);$   this implies that if $f$ is an {\sl odd} integer then the character $\bar{\chi}_n^f$ is {\sl nontrivial}.
\end{rem}

\begin{rem}
\label{splitchar}
If $m$ is another positive integer that is  relatively prime to $n$ and $\fchar(K), $ then the map
$$\mu_n \times \mu_m \to \mu_{nm},  \ (\gamma_1,\gamma_2)\mapsto \gamma_1\gamma_2$$
is an isomorphism of groups (and even Galois modules). The natural map
$$\phi_{n,m}:\Z/nm\Z  \to \Z/n\Z \times \Z/m\Z, \ \ c+nm\Z \mapsto (c+n\Z,c+m\Z)$$ is a ring homomorphism and the group homomorphism
$$\bar{\chi}_{nm}: \Gu K\to (\Z/nm\Z)^{*}$$
coincides with
$$g \mapsto (\bar{\chi}_{n}(g),\bar{\chi}_{m}(g)) \in (\Z/n\Z)^{*}\times (\Z/m\Z)^{*}\stackrel{\phi_{n,m}^{-1}}{\longrightarrow} (\Z/nm\Z)^{*}.$$\end{rem}

If $A$
is an abelian variety over $K$ then we write $A[n]$ for the kernel
of multiplication by $n$ in $A(\bar{K})$. It is well known that $A[n]$ is a
finite Galois submodule of $A(\bar{K}).$ If we forget about the
Galois action then $A[n]$ is a free $\Z/n\Z$-module of rank
$2\,\dim(A)$.

The following assertion is stated without  proof, as an exercise,
in Serre's Lectures on the Mordell--Weil Theorem \cite[Sect. 4.6, p.
55]{SerreMW}.

\begin{thm}
\label{SerreEx}
If $\dim(A)>0$ and $A[n]\subset A(K)$ then $\mu_n \subset K$.
\end{thm}

\begin{proof}
First, it suffices to check the case when $n=\ell^r$ is a power of a
prime $\ell \ne \fchar(K)$.

Second, if $A^t$ is the dual of $A$ then let us take a
$K$-polarization $\lambda: A \to A^{t}$ of smallest possible degree.
Then $\lambda$  is  not divisible by $\ell$, i.e., $\ker(\lambda)$
does not contain the whole $A[\ell]$. Otherwise,  divide
$\lambda$ by $\ell$ to get a $K$-polarization of lower degree.
Thus the image $\lambda(A[\ell^r])\subset A^t[\ell^r]$ contains a
point of exact order $\ell^r$, say $Q$. Otherwise,
$$\lambda(A[\ell^r])\subset A^t[\ell^{r-1}]$$
 and therefore $A[\ell]=\ell^{r-1}A[\ell^r]$ lies in the kernel of $\lambda$, which is not the case.

Since $A[\ell^r]\subset A[K]$ and $\lambda$ is defined over $K$, the
image $\lambda(A[\ell^r])$ lies in $A^t(K)$. In particular, $Q$ is a
$K$-rational point on $A^t$.

Third, there is a {\sl nondegenerate} Galois-equivariant Weil
pairing \cite{LangAV}
$$e_n: A[\ell^r] \times A^t[\ell^r] \to \mu_{\ell^r}.$$
I claim that there is a point $P \in A[\ell^r]$ such that $e_n(P,Q)$
is a primitive $\ell^r$th root of unity. Indeed, otherwise
$$e_n(A[\ell^r],Q) \subset \mu_{\ell^{r-1}}$$
so that  the nonzero point $\ell^{r-1}Q$ is orthogonal to the whole
$A[\ell^r]$ with respect to $e_n$, which contradicts the
nondegeneracy of $e_n$.

Thus, $\gamma:=e_n(P,Q)$
is a primitive $\ell^r$th root of unity that lies in $K$, because
both $P$ and $Q$ are $K$-points. Since  $\mu_{\ell^r}$ is
generated by $\gamma$, $\mu_{\ell^r}\subset K.$
\end{proof}

The aim of this paper is to a prove a variant of Serre's exercise that deals with the Galois action on the twisted odd-dimensional \'etale cohomogy group  with finite coefficients of a smooth projective variety (see Theorem \ref{mainRoot} below). Our proof is based on   the Hard Lefschetz Theorem  \cite{DeligneW2} and the unimodularity of Poincar\'e duality \cite{ZarhinP}.

\begin{sect}
\label{basicX}
 If  $\Lambda$  is a commutative ring with $1$ and without zero
divisors and $M$ is a $\Lambda$-module,  then we write $M_{\tors}$
for its torsion submodule and $M/\tors$ for the quotient
$M/M_{\tors}$. Usually, we will use this notation when $\Lambda$ is
 the ring $\Z_{\ell}$ of $\ell$-adic
integers.

If $\ell$ is a prime different from $\fchar(K)$ then we write $\Z_{\ell}(1)$ for the projective limit
of the cyclic Galois modules $\mu_{\ell^r}$ with  $\ell$th power as transition map. It is known that
$\Z_{\ell}(1)$  is a free $\Z_{\ell}$-module of rank $1$ with natural continuous action of $\Gu K$ defined by the {\sl cyclotomic character}
$$\chi_{\ell}: \Gu K \to \Aut_{\Z_{\ell}}(\Z_{\ell}(1))=\Z_{\ell}^{*}.$$
There are canonical isomorphisms
$$\Z_{\ell}/\ell^r \Z_{\ell}=\Z/\ell^r\Z, \ \Z_{\ell}(1)/\ell^r\Z_{\ell}(1)=\mu_{\ell^r};$$
in addition
$$\chi_{\ell}\bmod \ell^r=\bar{\chi}_{\ell^r}$$
for all positive integers $r$.   

We write $\Q_{\ell}(1)$ for the one-dimensional $\Q_{\ell}$-vector space 
$$\Q_{\ell}(1)=\Z_{\ell}(1)\otimes_{\Z_{\ell}}\Q_{\ell}$$
provided with the natural Galois action that is defined by the character $\chi_{\ell}$. For each integer $a$ we will need the $a$th tensor
power $\Q_{\ell}(a):=\Q_{\ell}(1)^{\otimes a}$, which is  a one-dimensional $\Q_{\ell}$-vector space provided with the  Galois action that is defined by the character $\chi_{\ell}^a$.

Let
$X$ be an absolutely irreducible  smooth projective variety over $K$ of positive dimension $d=\dim(X)$. We write $\bar{X}$ for the irreducible smooth projective $d$-dimensional variety $X \times_K \bar{K}$ over $\bar{K}$.
Let  $\ell$ be a prime $\ne
\fchar(K)$ and $a$ an integer. If  $i \le 2d$ is a nonnegative
integer then we write $H^i(\bar{X},\Z_{\ell}(a))$  for the corresponding twisted $i$th
\'etale $\ell$-adic cohomology group.
Recall that all the \'etale cohomology groups $H^i(\bar{X},\mu_n^{\otimes a})$ 
 are finite $\Z/n\Z$-modules
and that
 the $\Z_{\ell}$-modules $H^i(\bar{X},\Z_\ell(a))$  are finitely generated. In particular,  each  $H^i(\bar{X},\Z_{\ell}(a))/\tors$ is a free $\Z_{\ell}$-module of finite rank. These
finiteness results are fundamental finiteness theorems in \'etale
cohomology from {\bf SGA 4, 4$\frac{1}{2}$, 5}, see \cite{FK} and \cite[pp.
22--24]{Katz} for precise references.  All these groups are provided with the natural linear continuous actions of $\Gu K$. We also consider the corresponding finite-dimensional $\Q_{\ell}$-vector spaces
$$H^i(\bar{X},\Q_\ell(a))=H^i(\bar{X},\Z_\ell(a))\otimes_{\Z_{\ell}}\Q_{\ell}.$$
The Galois action on $H^i(\bar{X},\Z_\ell(a))$ extends by $\Q_{\ell}$-linearity to 
$H^i(\bar{X},\Q_\ell(a))$.  There are natural isomorphisms of $\Gu K$-modules
$$H^i(\bar{X},\Q_\ell(a+b))=H^i(\bar{X},\Q_\ell(a))\otimes_{\Q_{\ell}}\Q_{\ell}(b)$$
for all integers $a$ and $b$.  

\begin{rem}
\label{split}
If  a positive integer $m$  is relatively prime to $n$ and $\fchar(K)$, then the splitting $\mu_{nm}=\mu_n\times \mu_m$ induces the splitting of Galois modules
$$H^i(\bar{X},{\mu_{nm}}^{\otimes a})=H^i(\bar{X},{\mu_n}^{\otimes a})\oplus H^i(\bar{X},{\mu_m}^{\otimes a}).$$
\end{rem}

The $\Q_{\ell}$-dimension of $H^i(\bar{X},\Q_\ell(a))$ is denoted by $\bet_i(\bar{X})$ and called the $i$th {\sl Betti number} of $\bar{X}$: it does not depend on a choice of ($a$ and) $\ell$. In characteristic zero it follows from the comparison theorem between classical and \'etale cohomology \cite{Milne}. In finite characteristic the independence follows from results of Deligne \cite{DeligneW1}. It is also known that $\bet_i(\bar{X})=0$ if $i>2d$ \cite{Katz,FK}.
\end{sect}

Our main result is the following statement.

\begin{thm}
\label{mainRoot}
Let $i$ be a nonnegative integer.
\begin{itemize}
\item[(i)]
 Suppose that $i\le d-1$ and $\bet_{2i+1}(\bar{X})\ne 0$. If the Galois action on 
$H^{2i+1}(\bar{X},{\mu_n}^{\otimes i})$ is trivial then $\mu_n \subset K$.
\item[(ii)]
Suppose that $1\le i\le d$ and $\bet_{2i-1}(\bar{X})\ne 0$. If the Galois action on 
$H^{2i-1}(\bar{X},{\mu_n}^{\otimes i})$ is trivial then $\mu_n \subset K$.
\end{itemize}
\end{thm}

\begin{ex}
Let us take $i=1$. Then  {\sl Kummer theory} tells us that
$$H^{2i-1}(\bar{X},{\mu_n}^{\otimes i})=H^{1}(\bar{X},\mu_n)=\Pic(\bar{X})[n]$$
is the kernel of multiplication by $n$ in the Picard group $\Pic(\bar{X})$ of $\bar{X}$.  On the other hand if $B$ is an abelian variety over $K$ that is the Picard variety of $X$ \cite{LangAV} then $\dim(B)=\bet_1(\bar{X})$ and $B[n]$ is a Galois submodule of $H^{1}(\bar{X},\mu_n)$. If we know that the Galois action on $H^{1}(\bar{X},\mu_n)$ is trivial then the same is true for its submodule $B[n]$. Now if 
$\bet_1(\bar{X})\ne 0$ then $B \ne \{0\}$ and Theorem \ref{SerreEx} applied to $B$ implies that $\mu_n \subset K$. 
\end{ex}

Theorem \ref{mainRoot} may be viewed as a special case (when  $a =\frac{j \pm 1}{2}$) of the following statement.

\begin{thm}
\label{nGeneral}
Let $j$ be a nonnegative integer  and $\bet_j(\bar{X})\ne 0$. Let $a$ be an integer.  Assume that the Galois action on $H^j(\bar{X},{\mu_n}^{\otimes a})$ is trivial.  Then
$$\bar{\chi}_{n}^{2a-j}(g)
=1 \ \forall g \in G=\Gu K.$$
If, in addition,  $2a-j$ is relatively prime to 
$\phi(n)$ then $\mu_n \subset K$.
\end{thm}

\begin{cor}[Corollary to Theorem \ref{nGeneral}]
\label{converse}
Let $K$ be a field, $n$ a positive integer prime to  $\fchar(K)$. Suppose that $K$ does not contain a primitive $n$th root of unity.  
Suppose that $j$ is an odd positive integer.  Let $a$ be an integer such that 
 $2a-j$ is relatively prime to 
$\phi(n)$. 
 Then for each 
absolutely irreducible smooth projective  variety $X$ over $K$ with $\bet_j(\bar{X})\ne 0$ the Galois group $\Gu K$ acts nontrivially on
$H^j(\bar{X},\mu_n^{\otimes a})$ 
\end{cor}

The next assertion covers (in particular) the case of {\sl quadratic}  $\bar{\chi}_n$ (e.g., when $K$ is the maximal real subfield $\Q(\mu_n)^{+}$ of the $n$th cyclotomic field $\Q(\mu_n)$ of $\Q$.)

\begin{thm}
\label{converse2}
Let $K$ be a field, $n$ a positive integer prime to   $\fchar(K)$.
Suppose that  the degree 
$[K(\mu_n):K]$ is even. (E.g.,
 $K(\mu_n)/K$ is a quadratic extension.)   Then for each positive odd integer $j$, each integer $a$ and every
absolutely irreducible smooth projective  variety $X$ over $K$ with $\bet_j(\bar{X})\ne 0$ the Galois group $\Gu K$ acts nontrivially on
$H^j(\bar{X},\mu_n^{\otimes a})$.
\end{thm}

\begin{rem}
The special case of Theorem \ref{converse2} when $\bar{\chi}_n$ is a quadratic character follows directly from Theorem \ref{mainRoot}, because in this case  the Galois module $H^j(\bar{X},\mu_n^{\otimes a})$  is isomorphic either to $H^j(\bar{X},\mu_n^{\otimes [(j+1)/2]})$  or to $H^j(\bar{X},\mu_n^{\otimes [(j-1)/2]})$.
\end{rem}

The paper is organized as follows. Section \ref{algebra} contains auxiliary results about pairings between finitely generated modules over discrete valuation rings. We use them in Section \ref{proofs}, in order to prove Theorems \ref{nGeneral}, \ref{mainRoot} and \ref{converse2}.

{\bf Acknowledgements}.
This work is a follow up of \cite{ZarhinIzv1982,ZarhinCamb99,ZarhinP}.  I am deeply grateful to Alexey Parshin for helpful comments. My thanks also go to
Peter Schneider, Alice Silverberg, Adebisi Agboola, 
 Alexei Skorobogatov and Nick Katz for their interest in this paper (and/or the previous ones).  My very special thanks go to the referee, whose numerous suggestions and comments significantly improved the exposition.
 The final version of this paper was prepared during my stay at the Max-Planck-Institut f\"ur Mathematik (Bonn), whose hospitality and support are gratefully acknowledged.

\section{Linear algebra}
\label{algebra}

This section contains auxiliary results that will be used in the next section in order to prove main results of the paper.

\begin{sect}
\label{general}
Let $E$ be a  discrete valuation field, $\Lambda \subset E$
the corresponding discrete valuation ring with maximal ideal $\mm$.
Let $\pi \in \mm$ be an uniformizer, i.e., $\mm=\pi\Lambda$.

If $U$ is a finitely generated $\Lambda$-module then we write $U_E$ for the corresponding
finite-dimensional  $E$-vector space $U\otimes_{\Lambda}E$. The kernel of the homomorphism of $\Lambda$-modules
$$\otimes 1: U \to U\otimes_{\Lambda}E=U_E, \ x \mapsto x\otimes 1$$
coincides with $U_{\tors}$ while the image 
$$\tilde{U}:=\otimes 1(U)\subset U_E$$
is a $\Lambda$-lattice in $V_E$ of maximal rank $\dim_E(U_E)$.

Let $G$ be a  group and
$$\chi: G \to \Lambda^{*}\subset E^{*}$$
is a homomorphism  of $G$ to the group $\Lambda^{*}$ of invertible elements of $\Lambda$. If 
$H$ is a nonzero finite-dimensional vector space over $E$ and
$$\rho: G \to \Aut_E(H)$$
is a $E$-linear representation of $G$ in $H$ then $H$ becomes a module over the group algebra
 $E[G]$ of $G$ over $E$. Then
$$\rho\otimes \chi: G \to \Aut_E(H),   \ \rho\otimes \chi(g)=\chi(g)\rho(g) \ \forall g \in G$$
is also a linear representation  of $G$ in $H$. We denote the corresponding $E[G]$-module by $H(\chi)$ and call it the {\sl twist}
of $H$ by $\chi$. Notice that $H$ and  $H(\chi)$ coincide as $E$-vector spaces. It is also clear that if $T$ is a $\Lambda$-lattice in $H$ then
it is $G$-stable in $H(\chi)$ if and only if it is $G$-stable in  the $E[G]$-module $H$. On the other hand, let $L$ be a {\sl one-dimensional} $E$-vector space provided with a structure of $G$-module defined by
$$g z:= \chi(g)z \ \forall g \in G, z \in L.$$
Then the $G$-modules $H(\chi)$ and $H\otimes_E L$ are isomorphic noncanonically.
\end{sect}

\begin{lem}
\label{balanced}
Suppose that $H_1$ and $H_2$  are nonzero finite-dimensional $E$-vector spaces and
$$\rho_1: G \to \Aut_E(H_1), \  \rho_2: G \to \Aut_E(H_2)$$
are  
isomorphic
 $E$-linear representations of $G$.  
Suppose that $T_1$ is a $G$-stable $\Lambda$-lattice in $H_1$ of rank $\dim_E(H_1)$ and
$T_2$ is a $G$-stable $\Lambda$-lattice in $H_2$ of rank $\dim_E(H_2)$. Then there is an isomorphism of $E[G]$-modules
$u:  H_1 \to H_2$
such that
$$u(T_1)\subset T_2, \ u(T_1) \not\subset \pi\cdot T_2.$$
\end{lem}

\begin{proof}
Clearly,
$$H_2 =\bigcup_{j=1}^{\infty} \pi^{-j}\cdot T_2, \   \bigcap _{j=1}^{\infty} \pi^{j}\cdot T_2=\{0\}.$$
Let $u_0: H_1 \cong H_2$ be an isomorphism of $E[G]$-modules. Since $H_1$ is a finitely generated $\Lambda$-module, there exists an integer $j$ such that
$\pi^{-j}\cdot u_0(T_1)\subset T_2$. Let us take the smallest $j$ that enjoys this property and put $u=\pi^{-j}u_0$.
\end{proof}

\begin{thm}
\label{mainLinear}
Suppose that $U$ and $V$ are finitely generated $\Lambda$-modules provided with  group homomorphisms
$$G \to \Aut_{\Lambda}(U), \  G \to \Aut_{\Lambda}(V).$$
Let us assume that $ U/\tors \ne \{0\}$,  i.e., rank of $U$ is positive.

Suppose that we are given a $\Lambda$-bilinear pairing
$$e: U \times V \to \Lambda$$ that enjoys the following properties.

\begin{itemize}
\item[(i)]
$$e(gx,gy)=\chi(g)\cdot e(x,y) \ \forall g\in G; x\in U, y\in V.$$
\item[(ii)]
The $\Lambda$-bilinear pairing
$$U/\tors \times V/\tors \to \Lambda$$
induced by $e$ is {\sl perfect} (unimodular).
\item[(iii)]
The $E[G]$-modules $U_E$ and $V_E$ are isomorphic.
\end{itemize}

Let  $r$ be a positive integer such that the induced $G$-action on $U/\pi^r U$ is trivial, i.e.,
$$x-gx \in \pi^r U \ \forall g\in G, x\in U.$$

Then
$$\chi(g) \bmod \pi^r\Lambda=1 \in \Lambda/\pi^r\Lambda \ \forall g \in G.$$
\end{thm}

\begin{proof}
Clearly,
$$e(U_{\tors},V) =\{0\}=e(U,V_{\tors}).$$
Also  $U_{\tors}$ is a $G$-submodule of $U$ and $V_{\tors}$ is a $G$-submodule of $V$. Moreover, the $G$-module  $[U/{\tors}]/\pi^r [U/{\tors}]$
is isomorphic to a quotient of the $G$-module $U/\pi^r U$. In particular, the $G$-action on  $[U/{\tors}]/[\pi^r U/{\tors}]$ is also trivial. In the notation of Sect. \ref{general},
the natural homomorphisms
$$U/{\tors}=U/U_{\tors} \to \tilde{U},  \ x+U_{\tors} \mapsto x\otimes 1,  \ V/{\tors}=V/V_{\tors} \to \tilde{V},  \ x+V_{\tors} \mapsto x\otimes 1$$
are $G$-equivariant isomorphisms of free $\Lambda$-modules of finite rank
$$U/{\tors} \cong \tilde{U}, \ V/{\tors} \cong \tilde{V}$$
where $\tilde{U}$  and $\tilde{V}$ are $G$-stable lattices of maximal rank in $U_E$ and $V_E$ respectively. This implies that  the $G$-action on 
$\tilde{U}/\pi^r \tilde{U}$ and $e$ induces a $\Lambda$-bilinear {\sl perfect} pairing
$$\tilde{e}: \tilde{U} \times \tilde{V}\to \Lambda$$
such that
$$\tilde{e}(gx,gy)=\chi(g)\cdot \tilde{e}(x,y) \ \forall g\in G; x\in  \tilde{U}, y\in \tilde{V}.$$

Applying Lemma \ref{balanced} to the isomorphic $E[G]$-modules $U_E$ and $V_E$,
we obtain a ``nicer" isomorphism of $E[G]$-modules $u: U_E \cong V_E$ such that
$$u(T_1)\subset T_2, \ u(T_1) \not\subset \pi T_2.$$
Let us pick $x_0 \in T_1$ with $y:=u(x_0)\not\in \pi T_2$. Since $x_0\bmod \pi^r T_1 \in T_1/ \pi^r T_1$ is $G$-invariant,
its image 
$$u(x)\bmod \pi^r T_2=  y\bmod \pi^r T_2 \in T_2/\pi^r T_2$$
 is also $G$-invariant.  Since $y$ is {\sl not} divisible in $T_2$,
the $\Lambda$-submodule $\Lambda \cdot y$ is a direct summand of $T_2$.
Since the pairing $\tilde{e}$ between $T_1$ and $T_2$ is perfect, there is $x \in T_1$ with $e(x,y)=1$. This implies that
$$\chi(g)=\chi(g)\cdot 1=\chi(g)\cdot \tilde{e}(x,y)=\tilde{e}(gx,gy),$$
i.e.,  
$$\chi(g)=\tilde{e}(gx,gy) \ \forall g \in G.$$
On the other hand, since 
$$x -gx \in \pi^r T_1, \  y -gy \in \pi^r T_2,$$
we have
$$\tilde{e}(gx,gy)-\tilde{e}(x,y) \in \pi^r \Lambda \ \forall g \in G.$$
This means that
$$\chi(g)-1=\tilde{e}(gx,gy)-\tilde{e}(x,y)  \in \pi^r \Lambda \  \forall g \in G$$
and we are done.
\end{proof}

The next statement is a useful variant of Theorem \ref{mainLinear} that deals with {\sl twisted} representations.

\begin{thm}
\label{mainLinear2}
Suppose that $U$ and $V$ are finitely generated $\Lambda$-modules provided with  group homomorphisms
$$G \to \Aut_{\Lambda}(U), \  G \to \Aut_{\Lambda}(V).$$
Assume that $ U/\tors \ne \{0\}$,  i.e. the  rank of $U$ is positive.

Suppose that we have a $\Lambda$-bilinear pairing
$$e: U \times V \to \Lambda$$ that enjoys the following properties.

\begin{itemize}
\item[(i)]
$$e(gx,gy)= e(x,y) \ \forall g\in G; x\in U, y\in V.$$
\item[(ii)]
The $\Lambda$-bilinear pairing
$$U/\tors \times V/\tors \to \Lambda$$
induced by $e$ is {\sl perfect} (unimodular).
\item[(iii)]
The $E[G]$-modules $U_E$ and $V_E(\chi)$ are isomorphic.
\end{itemize}

Let  $r$ be a positive integer  such that the induced $G$-action on $U/\pi^r U$ is trivial, i.e. 
$$x-gx \in \pi^r U \ \forall g\in G, x\in U.$$

Then
$$\chi(g) \bmod \pi^r\Lambda=1 \in \Lambda/\pi^r\Lambda \ \forall g \in G.$$
\end{thm}

\begin{proof}
Let 
$$\rho_U: G \to \Aut_{\Lambda}(U), \ \rho_V: G \to \Aut_{\Lambda}(V)$$ be the structure homomorphisms that define the actions of $G$ on $U$ and  $V$ respectively.  In this notation,
$$e(\rho_U(g)x,\rho_V(g)y)= e(x,y) \ \forall g\in G; x\in U, y\in V.$$
Let us twist $\rho_V$  by considering the group homomorphism
$$\rho_{V(\chi)}:G \to \Aut_{\Lambda}(V), \ g \mapsto \chi(g)\rho(g).$$
We denote the resulting $G$-module by $V(\chi)$ and call it the {\sl twist} of $V$ by $\chi$. Notice that $V$ coincides with $V(\chi)$ as $\Lambda$-module. On the other
hand, the $E[G]$-module $V(\chi)_E$ is canonically isomorphic to $V_E(\chi)$. The pairing $e$ defines the $\Lambda$-bilinear pairing
$$e_{\chi}: U \times V(\chi) \to \Lambda, \ e_{\chi}(x,y):=e(x,y) \ \forall x\in U, y\in V=V(\chi)$$
of $G$-modules $U$ and $V(\chi)$,
which satisfies
$$e_{\chi}(\rho_U(g)x, \rho_{V(\chi)}(g)y)=
e(\rho_U(g)x,\chi(g)\rho_V(g)y)=
\chi(g) e(\rho_U(g)x,\rho_V(g)y)=$$
$$\chi(g) e(x,y)=
\chi(g)e_{\chi}(x,y) \ \forall g \in G; x \in U, y \in V(\chi).$$
This implies that
$$e_{\chi}(\rho_U(g)x, \rho_{V(\chi)}(g)y)=\chi(g)e_{\chi}(x,y) \ \forall g \in G; x \in U, y \in V(\chi).$$
Now the result follows from Theorem \ref{mainLinear} applied to $U$, $V(\chi)$ and $e_{\chi}$.
\end{proof}

\section{Proofs of main results}
\label{proofs}
Let $\ell$ be a prime different from $\fchar(K)$ and $r$ a positive integer.
 Let us put $$E=\Q_{\ell}, \Lambda=\Z_{\ell},
\pi=\ell,  G=\Gu K.$$   We keep the notation and assumptions of Sect. \ref{basicX}.
Recall that $d=\dim(X)\ge 1$.

\begin{prop}
\label{ajGeneral}
Let $j$ be a nonnegative integer with $j \le 2d$ and $\bet_j(\bar{X})\ne 0$. Let $a$ be an integer.  Assume that the Galois action on $H^j(\bar{X},{\mu_{\ell^r}}^{\otimes a})$ is trivial.  Then
$$\bar{\chi}_{\ell^r}^{2a-j}(g)
=1 \ \forall g \in G=\Gu K.$$
\end{prop}

\begin{proof}
Let us put
$U:=H^j(\bar{X},\Z_\ell(a))$: it is provided with the natural structure of $G=\Gu K$-module. Then the universal coefficients theorem \cite[Ch. V, Sect. 1, Lemma 1.11]{Milne}
gives us a canonical $\Gu K$-equivariant embedding
$$U/\ell^r U=H^j(\bar{X},\Z_\ell(a))/\ell^r H^j(\bar{X},\Z_\ell(a)) \hookrightarrow H^j(\bar{X},{\mu_n}^{\otimes a}).$$
Since the Galois action on $H^j(\bar{X},{\mu_n}^{\otimes a})$ is trivial,  it is also trivial on $U/\ell^r U$.
We have (in the notation of Sect. \ref{general})
$$U_E=H^j(\bar{X},\Z_\ell(a))\otimes_{\Z_{\ell}}\Q_{\ell}=H^j(\bar{X},\Q_\ell(a)).$$ 
Let 
$V:=H^{2d-j}(\bar{X},\Z_\ell(d-a))$: it has  the natural structure of $G=\Gu K$-module and
$$V_E=H^{2d-j}(\bar{X},\Z_\ell(d-a))\otimes_{\Z_{\ell}}\Q_{\ell}=H^{2d-j}(\bar{X},\Q_\ell(d-a)).$$
The cup product pairing gives rise to a $\Z_{\ell}$-bilinear $\Gu K$-invariant pairing  known as {\sl Poincar\'e duality} (\cite[Ch.
VI, Sect. 11, Cor. 11.2 on p. 276]{Milne}, \cite[p. 23]{Katz}, \cite[Ch. II, Sect. 1]{FK})
$$e: H^j(\bar{X},\Z_\ell(a))\times H^{2d-j}(\bar{X},\Z_\ell(d-a)) \to  \to
H^{2d}(\bar{X},\Z_{\ell}(d))\cong \Z_{\ell}.$$
It is known \cite{ZarhinP} that the induced pairing of free $\Z_{\ell}$-modules of finite rank
$$e: H^j(\bar{X},\Z_\ell(a))/{\tors}\times H^{2d-j}(\bar{X},\Z_\ell(d-a))/{\tors} \to \Z_{\ell}$$
is {\sl perfect and unimodular}.

Let us choose an invertible very ample sheaf $\mathcal{L}$   on $X$ and let 
$$h \in H^2(\bar{X},\Q_{\ell}(1))^{\Gu K}\subset H^2(\bar{X},\Q_{\ell}(1))$$ be its first $\ell$-adic Chern class. 
If $j\le d$ then the Hard Lefschetz Theorem (\cite{DeligneW2}, \cite[Ch. IV, Sect. 5, pp. 274--275]{FK})  tells us that cup multiplication by $(d-j)$th power of $h$ establishes an isomorphism between the $\Q_{\ell}$-vector spaces
$H^j(\bar{X},\Q_{\ell}(a))$ and $H^{2d-j}(\bar{X},\Q_{\ell}(a+d-j))$.  On the other hand, if $d \ge j$  then cup multiplication by the $(j-d)$th power of $h$ 
establishes an isomorphism between $\Q_{\ell}$-vector spaces
 $H^{2d-j}(\bar{X},\Q_{\ell}(a+d-j))$ and $H^j(\bar{X},\Q_{\ell}(a))$. In both cases the Galois-invariance of $h$ implies that the $\Q_{\ell}$-vector spaces
$U_E=H^j(\bar{X},\Q_{\ell}(a))$ and $H^{2d-j}(\bar{X},\Q_{\ell}(a+d-j))$ are isomorphic as $\Gu K$-modules.  On the other hand, the $\Gu K$-module 
$$H^{2d-j}(\bar{X},\Q_{\ell}(a+d-j))=H^{2d-j}(\bar{X},\Q_{\ell}(d-a+2a-j))=$$
$$H^{2d-j}(\bar{X},\Q_{\ell}(d-a))\otimes_{\Q_{\ell}} \Q_{\ell}(2a-j)\cong V_E(\chi)$$
where 
$$\chi:=\chi_{\ell}^{2a-j}: G=\Gu K \to \Z_{\ell}^{*}=\Lambda^{*}.$$
So the $G$-module $U_E$ is isomorphic to $V_E(\chi)$ and   Theorem \ref{mainLinear2} tells us  that
$$\bar{\chi}_{\ell^r}^{2a-j}(g)=(\chi_{\ell}(g))^{2a-j}\bmod  \ell^r\Z_{\ell}=\chi(g)\bmod \ell^r\Z_{\ell}=1 \ \forall g \in G=\Gu K.$$
\end{proof}

\begin{proof}[Proof of Theorem \ref{nGeneral}]
Since $\bet_j(\bar{X})\ne 0$, we have $j \le 2d$.
 Recall that $n$ is a positive integer that is not  divisible by $\fchar(K)$. Let $\ell$ be a prime dividing $n$ and let  $\ell^{r_n(\ell)}$ be the exact power of $\ell$ that divides $n$.  Applying Proposition  \ref{ajGeneral} to all such $\ell$ with $r=r_n(\ell)$ and using Remarks \ref{splitchar} and \ref{split}, we obtain that the character $\bar{\chi}_n^{2a-j}$ is trivial, which gives as the first  assertion  of Theorem \ref{nGeneral}. On the other hand, we know that $\bar{\chi}_n^{\phi(n)}$ is trivial. This implies that if
 $2a-j$ and $\phi(n)$ are relatively prime then $\bar{\chi}_n$ is itself trivial, i.e., $\mu_n\subset K$. This proves the second assertion of Theorem \ref{nGeneral}.
 \end{proof}


Now we use Theorem  \ref{nGeneral}  in order to prove Theorems  \ref{converse}  and \ref{converse2}.


\begin{rem}
In the statement of Theorem \ref{nGeneral} we do not require that $j$ is {\sl odd} and therefore its immediate Corollary \ref{converse} remains true without this assumption.  
However, if we drop this assumption in Corollary \ref{converse} (while keeping all the other ones) and assume instead that  $j$ is {\sl even} then $2a-j$ is also even and therefore $\phi(n)$ is odd, because it is relatively prime to $2a-j$. This implies that $n=2$ and therefore $\fchar(K) \ne 2$ and $K$ does {\sl not} contain a primitive square root of unity, i.e., $K$ does {\sl not} contain $-1$, which is absurd. 
\end{rem}

\begin{rem}
The second assertion of Theorem \ref{nGeneral}  (and its proof) remains true (valid) if in its statement we replace $\phi(n)$ by its divisor 
$\exp(n,K)$.
\end{rem}

\begin{proof}[Proof of Theorem \ref{mainRoot}] Since $a=(j\pm 1)/2$,  the integer $2a-j=\pm 1$ is relatively prime to  $\phi(n)$. Now the result follows
from already proven Theorem \ref{nGeneral}.

\end{proof}

\begin{proof}[Proof of Theorem \ref{converse2}]
Suppose that the Galois action on $H^j(\bar{X},{\mu_n}^{\otimes a})$ is trivial for some absolutely irreducible smooth projective variety $X$ with $\bet_j(\bar{X})\ne 0$. 
By Theorem \ref{nGeneral},  the character $\bar{\chi}_{n}^{2a-j}$ is trivial.  On the other hand, since $f:=2a-j$ is {\sl odd} and  $[K(\mu_n):K]$ is {\sl even}, Remark \ref{even} tells us that $\bar{\chi}_{n}^{2a-j}$ is {\sl nontrivial}. 
  This gives us a desired contradiction.
\end{proof}

\end{document}